\documentclass[10pt]{amsart}

\newtheorem{lem}{Lemma}[section]
\newtheorem{thm}{Theorem}[section]

\newtheorem{obs}{Remark}[section]

\usepackage{amssymb}
\usepackage{amsmath}
\usepackage{amsfonts}
\usepackage{graphicx}

\numberwithin{equation}{section}

\usepackage[T1]{fontenc}
\usepackage{palatino}

\def\k0{\kappa_0}

\def\lgl{\langle}
\def\rgl{\rangle}
\def\bfe{{\mathsf{e}}}

\def\bfu{{\bf{u}}}
\def\bfx{{\bf{x}}}
\def\bfy{{\bf{y}}}
\def\bfo{{\bf{0}}}
\def\bfn{{\bf{n}}}
\def\bfc{{\bf{c}}}

\def\mR3{{\mathbb{R}^3}}

\begin{document}
\title[Anomalous Cascades]
{Anomalous dissipation and energy cascade in 3D inviscid flows}
\author{R. Dascaliuc}
\address{Department of Mathematics\\
University of Virginia\\ Charlottesville, VA 22904}
\author{Z. Gruji\'c}
\address{Department of Mathematics\\
University of Virginia\\ Charlottesville, VA 22904}
\date{\today}
\begin{abstract}
Adopting the setting for the study of existence and scale locality
of the energy cascade in 3D viscous flows in \emph{physical space}
recently introduced by the authors to 3D inviscid flows, it is shown
that the anomalous dissipation is -- in the case of decaying
turbulence -- indeed capable of triggering the cascade which then
continues \emph{ad infinitum}, confirming Onsager's predictions.

\end{abstract}
\maketitle



\section{Introduction}

The paper in hand concerns the phenomenon of `anomalous dissipation'
and existence of the energy cascade in 3D inviscid incompressible
flows described by the 3D Euler equations,
\[
 \bfu_t + (\bfu \cdot \nabla)\bfu = - \nabla p,
\]
supplemented with the incompressibility condition $\, \mbox{div} \,
\bfu = 0$. The vector field $\bfu$ represents the velocity of the fluid
and the scalar field $p$ the (internal) pressure; the density is set
to 1.

\medskip

It was conjectured by Onsager in 1949 \cite{On49} that ``...in three
dimensions a mechanism for complete dissipation of all kinetic
energy, even without aid of the viscosity, is available.'' More
precisely, Onsager conjectured that the minimal spatial regularity
of a (weak) solution to the 3D Euler equations needed to conserve
energy is $\bigl(\frac{1}{3}\bigr)^+$, and that in the case the
energy is not conserved, the energy dissipation due to the lack of
regularity -- \emph{the anomalous dissipation} -- triggers the
energy cascade that continues \emph{ad infinitum} (cf. a summary of
Onsager's published and unpublished contributions to turbulence by
Eyink and Sreenivasan \cite{ES06}). In fact, as noticed in
\cite{ES06}, the following quotation from Onsager's note to Lin
(1945) seems to contain the first use of the word \emph{cascade} in
the theory of turbulence, ``The selection rule for the `modulation'
factor in each term of (8) suggests a `cascade' mechanism for the
process of dissipation, and also furnishes a dynamical basis for an
assumption which is usually made on dimensional grounds only''.

\medskip

There has been a series of mathematical works pertaining
$\frac{1}{3}$ minimality (most notably, the papers by Eyink
\cite{E94}, Constantin, E. and Titi \cite{CET94} and Duchon and
Robert \cite{DR00}) culminating with the paper by Cheskidov,
Constantin, Friedlander and Shvydkoy \cite{CCFS08} giving a
solution to one direction in $\frac{1}{3}$-minimality conjecture,
namely, showing
that as long as a weak solution to the 3D Euler equations is in the
space $L^3(0,T; B^{\frac{1}{3}}_{3, c_0})$ -- $B^{\frac{1}{3}}_{3,
c_0}$ being a subspace of Besov space $B^{\frac{1}{3}}_{3, \infty}$
in which boundedness over the Littlewood-Paley parameter $q$ is
replaced with the zero limit -- the energy equality holds. Looking
more precisely into local spatiotemporal structure and assuming that
the singular set is a smooth manifold, Shvydkoy \cite{Shvy09}
presented various spatiotemporal regularity criteria for the energy
conservation dimensionally equivalent to the critical one. In
addition, a complete solution to Onsager's $\frac{1}{3}$ minimality
in this setting was given for dyadic models (the models in which the
original non-local nonlinearity is replaced by a nonlinearity that
is local by design) (cf. \cite{CF09}). It may be tempting to think
that Onsager critical spatial regularity is also necessary for a
weak solution to conserve the energy; however, a family of explicit
energy-conserving flows well below Onsager criticality was recently
given by Bardos and Titi \cite{BT10} .

\medskip

On the other hand, to the best of our knowledge, there has been no
rigorous mathematical work showing that the anomalous dissipation is
indeed capable of triggering the energy cascade.

\medskip

Various methods of obtaining weak solutions to the Euler equations
were introduced in \cite{Sc93, Shn97, Shn00, DLS09, DLS10}. In
particular, the construction in \cite{Shn00} and a very recent work
\cite{DLS10} yield \emph{energy-dissipating} weak solutions to the
3D Euler.

\medskip

Recall that an $L^3_{loc}$ (in the space-time) solution $\bfu=(u^1,u^2,u^3)$
satisfies the \emph{local energy inequality} if
\[
 \partial_t \frac{1}{2} |\bfu|^2 + \, \mbox{div} \, \biggl( \Bigl(\frac{1}{2}
 |\bfu|^2+p\Bigr) \bfu \biggr)
 \le 0
\]
in the sense of distributions, i.e., if
\begin{equation}\label{lee}
 \frac{1}{2}\iint|\bfu|^2 \partial_t\phi \,d\bfx\,dt\
+ \iint\Bigl(\frac{1}{2}|\bfu|^2+p\Bigr)\bfu\cdot\nabla\phi\,d\bfx\,dt \ge 0
\end{equation}
for all nonnegative test functions $\phi$. (Note that
\[
 - \Delta p = \partial_i \partial_j u^i u^j
 \]
and the local elliptic theory imply that -- provided $\bfu$ is in
$L^3_{loc}$ -- all the terms in (\ref{lee}) are well-defined.)

\medskip

Let $\bfu$ be a weak solution to the 3D Navier-Stokes equations (NSE)
or an $L^3$ (in the space-time) weak solution to the 3D Euler
equations. (One should note that, at present, no general
construction of $L^3$ weak solutions to the 3D Euler equations
exists.) Duchon and Robert \cite{DR00} (in the case of the torus)
gave an explicit limit formula for a distribution $D(\bfu)$ (in the
space-time) measuring anomalous dissipation in the flow; defining
$D(\bfu)$ by
\[
 D(\bfu) = \lim_{\epsilon \to 0} \frac{1}{4} \int \nabla
 \phi^\epsilon(\bfy)
 \cdot \delta \bfu |\delta \bfu|^2 \, d\bfy
\]
where $\delta \bfu = \bfu(\bfx+\bfy)-\bfu(\bfx)$ and $\{\phi^\epsilon\}$ is a family
of standard mollifiers, the following form of the \emph{local energy
equality} holds,
\begin{equation}\label{du}
 \partial_t\Bigl(\frac{1}{2}|\bfu|^2\Bigr) + \, \mbox{div} \,
 \biggl(\Bigl(\frac{1}{2}|\bfu|^2+p\Bigr) \bfu\biggr) - \nu \Delta
 \frac{1}{2} |\bfu|^2 + \nu
 |\nabla \bfu|^2 + D(\bfu) = 0
\end{equation}
($\nu = 0$ for the Euler). Notice that in the case of the 3D NSE,
$D(\bfu) \ge 0$ (in the sense of distributions) is equivalent to the
local energy inequality; this is satisfied by all `suitable weak
solutions' constructed in \cite{Sc77, CKN82}, and in fact by any
weak solution obtained as a limit of a subsequence of the Leray
regularizations. In the case of the 3D Euler, $D(\bfu) \ge 0$ -- or
equivalently, the local energy inequality (\ref{lee}) holds -- for
any $L^3$ weak solution obtained as a strong $L^3$-limit of weak
solutions to the 3D NSE (as the viscosity $\nu$ goes to 0)
satisfying the local energy inequality. This motivated Duchon and
Robert to call weak solutions to the 3D incompressible fluid
equations satisfying $D(\bfu) \ge 0$ `dissipative'. Any
`dissipative' solution to the 3D Euler is also dissipative in the
sense of Lions \cite{PLL96}; a detailed proof of this fact can be
found in \cite{DLS10}, Appendix B. Note that we do not know if
globally dissipative solutions constructed in \cite{Shn00, DLS10}
are locally dissipative; there may be regions exhibiting local
creation of energy. However, at any given spatial scale $R_0$, there
will be regions exhibiting energy dissipation.

\medskip

In this paper, we show -- via suitable \emph{ensemble averaging} of
the local energy inequality over a region containing (possible)
singularities of the 3D Euler equations -- that provided the
anomalous dissipation in the region is strong enough (with respect
to the energy), the energy cascade commences and continues \emph{ad
infinitum}, confirming Onsager's predictions. In fact, in the case
of a spatially isolated singularity -- provided the anomalous
dissipation is positive, i.e., the strict energy inequality holds on
\emph{some} neighborhood of the singular curve -- the cascade
condition will hold on any small enough (in the spatial coordinates)
neighborhood of the singular curve.

\medskip

The approach is based on a very recent work \cite{DG10} in which a
setting for a rigorous mathematical study of the energy cascade in
\emph{physical space} was introduced. More precisely, the 3D NSE
were utilized via ensemble averaging of the local energy inequality
over a region of interest with respect to `$(K_1,K_2)$-covers' (see
below) to establish both existence of the energy cascade and scale
locality in decaying turbulence -- zero driving force and
non-increasing global energy -- under a very simple condition
plausible in the regions of intense fluid activity (large
gradients); namely, that Taylor micro scale is dominated by the
integral scale. This furnished the first proof of existence of the
energy cascade in 3D viscous flows in \emph{physical scales}, as
well as the only mathematical setting in which both existence of the
cascade and locality were obtained directly from the 3D NSE.

\medskip

For simplicity, assume that the region of interest is ball
$B(\bfo,R_0)$  ($R_0$ being the integral scale), $B(\bfo,2R_0)$
contained in the global spatial domain, and $0 < R \le R_0$. Let
$K_1$ and $K_2$ be two positive integers. A cover
$\{B(\bfx_i,R)\}_{i=1}^n$ of $B(\bfo,R_0)$ is a \emph{$(K_1,K_2)$
cover at scale $R$} if
\[
 \biggl(\frac{R_0}{R}\biggr)^3 \le n \le K_1
 \biggr(\frac{R_0}{R}\biggr)^3,
\]
and any point $\bfx$ in $B(\bfo,R_0)$ is covered by at most $K_2$ balls
$B(\bfx_i,2R)$; the parameters $K_1$ and $K_2$ represent
\emph{global} and \emph{local maximal multiplicities}, respectively.

\begin{obs}
\emph{The $(K_1,K_2)$-covers were originally named 'optimal
coverings' \cite{DG10}; here -- and elsewhere -- we renamed them
`$(K_1,K_2)$-covers` to emphasize a key role played by the maximal
multiplicities $K_1$ and $K_2$.}
\end{obs}

Let $f$ be an \emph{a priori} sign-varying density corresponding to
some physical quantity of interest (e.g., the flux density $ -
[(\bfu \cdot \nabla)\bfu+\nabla p] \cdot \bfu$), and consider the
arithmetic mean of the quantity per unit mass averaged over the
cover elements $B(\bfx_i, R)$,
\[
 F_R = \frac{1}{n} \sum_{i=1}^n \frac{1}{R^3} \int_{B(\bfx_i, 2R)} f \,
 \psi_i^\delta \, d\bfx,
\]
for some $0 < \delta \le 1$ where $\psi_i$ are smooth spatial
cut-offs associated with the balls $B(\bfx_i,R)$. The key property
of the ensemble averages $F_R$ is that $F_{R} \approx \,
\mbox{const} \, (R) \ $ for \emph{all} $(K_1,K_2)$-covers at scale
$R$ indicates there are no significant sign-fluctuations of the
density $f$ at scales comparable or greater than $R$. In other
words, if there are significant sign-fluctuations at scale $R^*$,
the averages $F_R$ will run over a wide range of values (across,
say, an interval $(-M,M)$ for some large $M$) -- simply by
rearranging and stacking up the cover elements up to the maximal
multiplicities -- for any $R$ comparable or less than $R^*$. Hence,
the averages act as a \emph{coarse detector} of the
sign-fluctuations at scale $R$ (of course, the large the
multiplicities, the finer detection).

In the case of a signed quantity (e.g., the energy density or the
enstrophy density) -- for any scale $R$, $0 < R \le R_0$ -- the
averages $F_R$ are all comparable to each other and in particular,
to the simple average over the spatial integral domain.

\medskip

Let $\bfu$ be a weak solution to the 3D Euler equations satisfying the
local energy inequality (\ref{lee}) (dissipative in the sense of
\cite{DR00}). Given $\phi$ -- a smooth spatiotemporal cut-off over
$(0,2T) \times B(\bfx,2R)$ -- denote by
$\varepsilon_{\bfx,R}$ the \emph{anomalous
dissipation} due to (possible) singularities located in the support
of $\phi$

\begin{equation}\label{anomall}
\varepsilon_{\bfx_,R} =
\iint\Bigl(\frac{1}{2}|\bfu|^2+p\Bigr)\bfu\cdot\nabla\phi\,d\bfy\,dt +
\frac{1}{2}\iint|\bfu|^2 \partial_t\phi \,d\bfy\,dt\ \ge 0,
\end{equation}
and let $\varepsilon_0= \frac{1}{T}\frac{1}{R_0^3}
\varepsilon_{\bfo,R_0}$ indicate the spatiotemporal average of the
anomalous dissipation due to singularities in $(0,T)\times
B(\bfo,2R_0)$. Following the general idea of ensemble averaging with
respect to $(K_1,K_2)$-covers, consider the spatiotemporal ensemble
averages of the local \emph{anomalous} dissipation quantities
$\{\varepsilon_{\bfx_i, R}\}$,

\begin{equation}
\varepsilon_R=\frac{1}{n}\sum\limits_{i=1}^{n}
\frac{1}{T}\frac{1}{R^3}\varepsilon_{\bfx_i,R}.
\end{equation}
Note that $\varepsilon_{x_i,R} = \bigl(D_i(\bfu), \phi_i\bigr)$
where $D_i(\bfu)$ is Duchon-Robert distribution $D(\bfu)$ measuring
anomalous dissipation associated with the ball $B(\bfx_i,R)$.

\medskip

A key property exploited in the proof of existence of the energy
cascade in the viscous case (cf. \cite{DG10}) was that the ensemble
averages of the time-averaged local viscous dissipation quantities
per unit mass at any scale $R$, $0 < R \le R_0$, were comparable
with the spatiotemporal average of the viscous dissipation term
associated to the integral domain $B(\bfo, R_0)$; this was simply a
consequence of the enstrophy density being non-negative. The
following lemma -- to be proved in the subsequent section --
provides an analogous statement in the realm of anomalous
dissipation, and is a key technical ingredient in establishing the
cascade in the inviscid case.

\begin{lem}
Let $\{B(\bfx_i, R)\}_{i=1}^n$ be a $(K_1,K_2)$-cover of $B(\bfo,
R_0)$ at scale $R$. Then, there exists a constant $K = K(K_1, K_2)
> 1$ such that
\[
 \frac{1}{K} \varepsilon_0 \le \varepsilon_R \le K \varepsilon_0,
\]
for any $R$, $0 < R \le R_0$.
\end{lem}

\begin{obs}
\emph{As a matter of fact, the proof of the lemma can be easily
modified to show that the analogous property holds for \emph{any
non-negative distribution}; in light of this observation, the
statement of the lemma is simply a consequence of the non-negativity
of Duchon-Robert distribution $D(\bfu)$.}
\end{obs}

\medskip

In the viscous case (3D NSE), the term $\iint |\bfu|^2 \nu \Delta
\phi \, d\bfx \, dt \,$ in the local energy inequality furnished the
breaking mechanism restricting the inertial range (via the estimate
on $\Delta \phi$ with respect to the spatial scale -- for details
see \cite{DG10}); in the inviscid case, once it starts -- provided
that $\varepsilon_0$ is strong enough compared to the spatiotemporal
average of the energy associated with $B(\bfo, R_0)$ -- the cascade
will continue indefinitely (as expected).

\medskip

Assuming certain geometric properties of the singular set leads to
improved results. As an illustration, we show that in the case of a
spatially isolated singularity (the singular set being a curve
$\bfc=\bfc(t)$) -- provided the strict energy inequality holds on
\emph{some} neighborhood of the singular curve -- the anomalous
dissipation will in fact dominate the corresponding energy on any
small enough (in spatial directions) neighborhood of the singular
curve, triggering the cascade. The reason behind this phenomenon is
that -- in the case of a spatially isolated singularity -- the
anomalous dissipation over any family of nested tubular
neighborhoods containing the singular curve is constant. Let us note
that a natural Onsager critical space here is $(L^3_t
L^{4.5}_{\bfx})_{loc}$ (cf. \cite{Shvy09}); hence, a class of weak
solutions compatible with existence of the energy cascade in this
setting is $(L^3_t L^\alpha_{\bfx})_{loc}, \, 3 \le \alpha < 4.5$.

\medskip

Scale locality of the cascade manifests in several ways analogous to
the viscous case (cf. \cite{DG10}). We present locality of the
ensemble averages of the time-averaged fluxes at spatial scale $R$.
In particular, considering the dyadic case --  $r = 2^k R$ ($k$ an
integer) -- both ultraviolet and infrared locality propagate
\emph{exponentially} in  $k$ as predicted by turbulence
phenomenology.


\section{Localized energy and flux; anomalous dissipation and ensemble averages}

Let $\bfu$ be a weak solution to the 3D Euler equations satisfying
the local energy inequality on a spatiotemporal domain $\Omega
\times (0,2T)$ (for simplicity, assume that $\Omega$ contains the
origin), and let $R_0>0$ be such that $B(\bfo,2R_0)$ is contained in
$\Omega$; $B(\bfo, R_0)$ is our region of interest and $R_0$ the
integral scale in the problem. Choose
$\psi_0\in\mathcal{D}(B(\bfo,2R_0))$ satisfying
\begin{equation}\label{psi0}
0\le\psi_0\le 1,\quad\psi_0=1\ \mbox{on}\ B(\bfo,R_0)\;.
\end{equation}

\medskip

For  $T>0$, $\bfx_0\in B(\bfo,R_0)$ and $0<R\le R_0$ define
$\phi=\phi_{\bfx_0,T,R}(t,\bfx)=\eta(t)\psi(\bfx)$ to be used in the
local energy inequality (\ref{lee}) where $\eta=\eta_T(t)$
and $\psi=\psi_{\bfx_0,R}(\bfx)$
satisfy the following conditions,

\begin{equation}\label{eta_def}
\eta\in\mathcal{D}(0,2T),\quad 0\le\eta\le1,\quad\eta=1\ \mbox{on}\
(T/4,5T/4),\quad\frac{|\eta'|}{\eta^{\delta}}\le\frac{C_0}{T}\;
\end{equation}
for some $0 < \delta \le 1$;

\medskip

\noindent if $B(\bfx_0,R)\subset B(\bfo,R_0)$, then $\psi\in\mathcal{D}(B(\bfx_0,2R))$
with $0\le\psi\le\psi_0, \ \psi=1$ on $B(\bfx_0,R)$,

\medskip

\noindent and if $B(\bfx_0,R)\not\subset B(\bfo,R_0)$, then $\psi\in\mathcal{D}(B(\bfo,2R_0))$
with  $0\le\psi\le\psi_0, \ \psi=1$ on $B(\bfx_0,R)\cap B(\bfo,R_0)$ satisfying the following:

\medskip

\noindent $\psi=\psi_0$ on the part of the cone in $\mR3$ centered at
zero and passing through
$S(\bfo,R_0)\cap B(\bfx_0,R)$ between $S(\bfo,R_0)$ and $S(\bfo,2R_0)$,
and $\psi=0$ on $B(\bfo,R_0)\setminus B(\bfx_0,2R)$ and outside
the part of the cone in $\mR3$ centered at zero and passing through
$S(\bfo,R_0)\cap B(\bfx_0,2R)$ between $S(\bfo,R_0)$ and $S(\bfo,2R_0)$.

\medskip

Figure \ref{ball_fig} illustrates the definition of $\psi$ in the case $B(\bfx_0,R)$ is not
 entirely contained in
$B(\bfo,R_0)$.

\begin{figure}
  \centerline{\includegraphics[scale=1, viewport=188 449 493 669, clip] {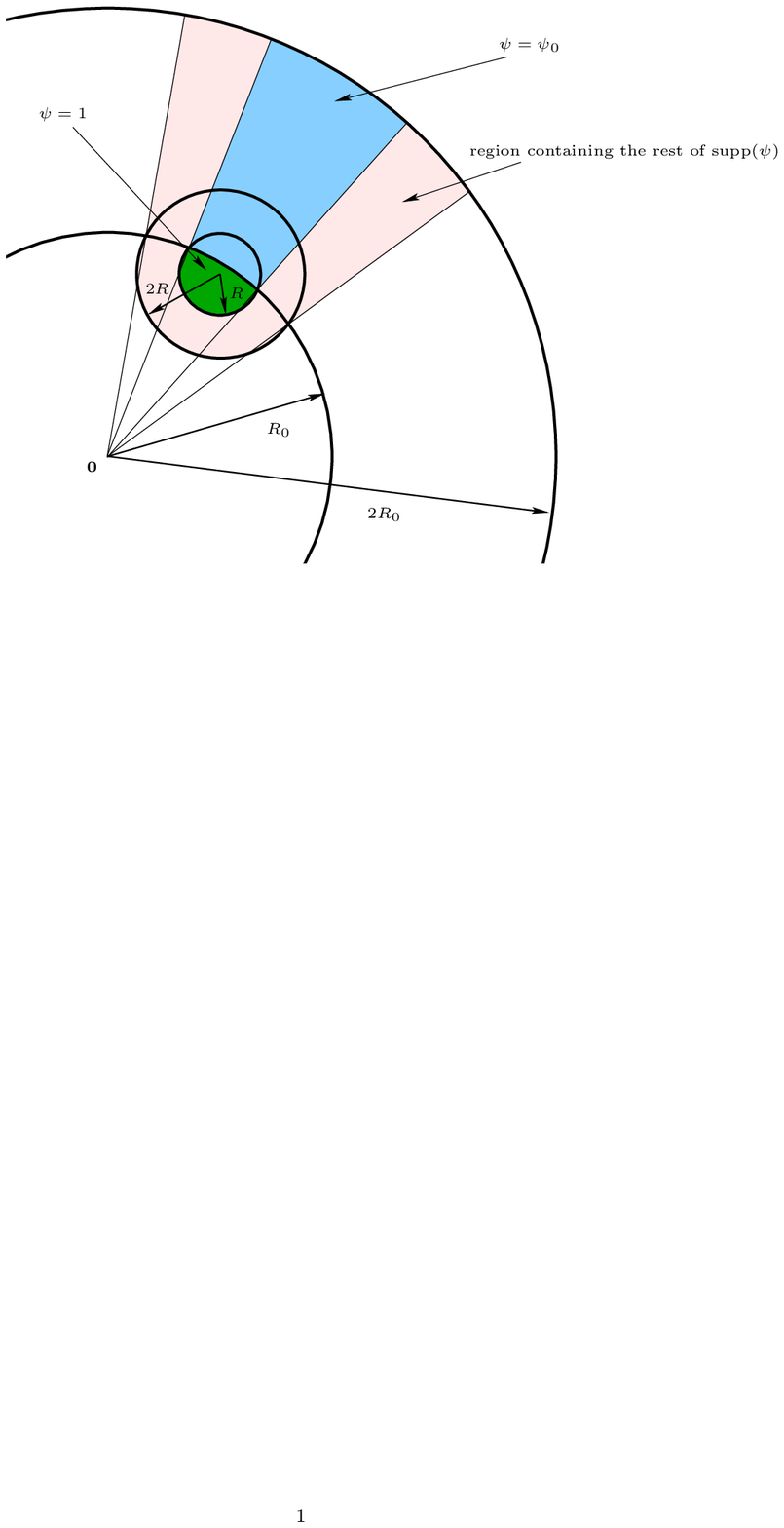}}
  \caption{Regions of supp$(\psi)$ in the case $B(\bfx_0,R)\not\subset B(\bfo,R_0)$,
  cross-section.}
  \label{ball_fig}
\end{figure}

\medskip

Note that in contrast to the Navier-Stokes case (cf. \cite{DG10}),
we do not make growth assumptions on derivatives of $\psi$. As we
shall see later, only the time derivatives of the test functions
will matter in the cascade formation.

\medskip

In the case $\bfx_{0}=\bfo$ and $R=R_{0}$ set
\begin{equation}\label{phi0}
\phi_0(t,\bfx)=\eta(t)\psi_0(\bfx).
\end{equation}

\medskip

Let $\bfx_0\in B(\bfo,R_0)$ and $0<R\le R_0$. Define localized
energy $\bfe$ at time $t$ associated with
$B(\bfx_0,R)$ by
\begin{equation}\label{enerdef}
\bfe_{\bfx_0,R}(t)=\int \frac{1}{2}|\bfu|^2\phi^{\delta}\,d\bfx\;.
\end{equation}

\medskip

A total inward flux -- (kinetic) energy plus pressure -- through the
boundary of a region $D$ is given by
\[
- \int\limits_{\partial D}\biggl(\frac{1}{2}|\bfu|^2+ p\biggr)\bfu
\cdot \bfn \,ds  \, = \, - \int\limits_{D}\Bigl((\bfu \cdot
\nabla)\bfu + \nabla p\Bigr) \cdot \bfu \, d\bfx
\]
where $\bfn$ is an outward normal. Considering the Euler equations
localized to
$B(\bfx_0,2R)$ -- and utilizing $\, \mbox{div} \, \bfu = 0$ -- leads to a
localized flux,
\begin{equation}\label{fluxdef}
\Phi_{\bfx_0,R}(t)=\int \biggl(\frac{1}{2}|\bfu|^2+p\biggr) \bfu \cdot \nabla
\phi \, d\bfx = - \int \Bigl((\bfu \cdot \nabla)\bfu + \nabla p\Bigr) \cdot \bfu
\, \phi \, d\bfx.
\end{equation}
Since $\psi$ can be constructed such that $\nabla \phi=\eta \,
\nabla \psi$ is oriented along the radial directions of $B(\bfx_0,
2R)$ \emph{toward the center of the ball}, $\Phi_{\bfx_0,R}$
represents the total flux {\em into} $B(\bfx_0,R)$ through the layer
between the spheres $S(\bfx_0,2R)$ and $S(\bfx_0,R)$ ($\nabla \phi
\equiv 0$ on $B(\bfx_0, R)$). (In the case of the boundary elements,
$\psi$ is almost radial and $\nabla\psi$ still points inward.)

\medskip

A more dynamic physical significance of the sign of $\Phi_{\bfx_0,
R}$ can be seen from the equations; for the sake of a more concise
interpretation of the local flux term $\Phi_{\bfx_0, R}$, let us for
a moment assume smoothness. Then, multiplying the Euler equations by
$\psi \bfu$ and integrating over $B(\bfx_0, 2R)$ leads to
\begin{equation}\label{loc_tr}
\frac{d}{dt} \int \frac{1}{2} |\bfu|^2 \psi \, d\bfx = \Phi_{\bfx_0, R}.
\end{equation}
Plainly, the positivity of $\Phi_{\bfx_0, R}$ implies the
increase of the kinetic energy around the point $\bfx_0$ at scale $R$.

\medskip

The key question is what can we say about the \emph{transfer of the
kinetic energy around the point $\bfx_0$ at scale $R$}, i.e., the
total exchange between the kinetic energy associated with the ball
$B(\bfx_0,R)$ and the kinetic energy in the complement of the ball
$B(\bfx_0,2R)$. In general, not much -- one can envision a variety
of scenarios. However, in physical situations where the kinetic
energy on the global spatial domain $\Omega$ is non-increasing
(here, we are concerned with the case of \emph{decaying turbulence},
setting the driving force to zero), the positivity of
$\Phi_{\bfx_0,R}$ implies the transfer of the kinetic energy around
the point $\bfx_0$ -- \emph{from larger scales} -- simply because
the local kinetic energy is increasing while the global kinetic
energy is non-increasing resulting in decrease of the kinetic energy
in the complement.

\medskip

Henceforth, following the discussion in the preceding paragraphs --
in the setting of decaying turbulence (zero driving force,
non-increasing global energy) -- the positivity and the negativity
of $\Phi_{\bfx_0, R}$ will be interpreted as transfer of kinetic
energy around the point $\bfx_0$ at scale $R$ toward smaller scales
and transfer of kinetic energy around the point $\bfx_0$ at scale
$R$ toward larger scales, respectively.

\medskip

As in Introduction, denote by $\varepsilon_{\bfx_0,R}$ the anomalous
dissipation of energy due to (possible) singularities inside $\,
\mbox{supp} \,  (\phi)$,

\begin{equation}\label{anomal-diss}
\varepsilon_{\bfx_{0},R}=\iint\frac{1}{2}|\bfu|^2\partial_t\phi\,d\bfx\,dt
+\iint\Bigl(\frac{1}{2}|\bfu|^2+p\Bigr)\bfu\cdot\nabla\phi\,d\bfx\,dt \ge 0\;.
\end{equation}

\medskip

For a physical quantity $\Theta_{\bfx,R}(t)$, $t\in(0,2T)$ and a
cover $\{B(\bfx_i,R)\}_{i=1,n}$ of $B(\bfo,R_0)$, consider the time
average of the ensemble-averaged local quantities
$\Theta_{\bfx_i,R}$ per unit mass,
\begin{equation}
\lgl\Theta\rgl_R=\frac{1}{T}\int
\frac{1}{n}\sum\limits_{i=1}^{n}
\frac{1}{R^3}\Theta_{\bfx_i,R}(t)\,dt\;.
\end{equation}

Set
\begin{equation}\label{e_R_def}
\bfe_R=\lgl \bfe_{\bfx,R}(t)\rgl_R\;
\end{equation}
and
\begin{equation}\label{Phi_R_def}
\Phi_R=\lgl \Phi_{\bfx,R}(t)\rgl_R\;;
\end{equation}
the averaged energy and inward-directed flux, respectively.

\medskip

In addition, consider the ensemble-averaged local anomalous
dissipation quantities $\varepsilon_{\bfx_i,R}$ -- per unit time and
per unit mass --

\begin{equation}\label{E_R_def}
\varepsilon_R=\frac{1}{n}\sum\limits_{i=1}^{n}
\frac{1}{T}\frac{1}{R^3}\varepsilon_{\bfx_i,R}\;.
\end{equation}

Finally, introduce the spatiotemporal average of the energy
associated to $B(\bfo,2R_0) \times (0,2T)$,
\begin{equation}\label{e_def}
\bfe_0=\frac{1}{T}\int
\frac{1}{R_0^3}\bfe_{\bfo,R_0}(t)\,dt=\frac{1}{T}\frac{1}{R_0^3}\iint \frac{1}{2}|\bfu|^2
\phi_0^\delta\,d\bfx\,dt\;
\end{equation}
and the spatiotemporal average of the anomalous dissipation on
$(0,2T)\times B(\bfx_0,2R)$,
\begin{equation}\label{E_def}
{\varepsilon_0}=\frac{1}{T}
\frac{1}{R_0^3}\varepsilon_{\bfo,R_0}=\frac{1}{T}\frac{1}{R_0^3}
\iint\frac{1}{2}\Bigl(|\bfu|^2\partial_t\phi_{0}
+(|\bfu|^2+2p)\bfu\cdot\nabla\phi_{0}\Bigr) \,d\bfx\,dt\;
\end{equation}
with $\phi_0$ defined in (\ref{phi0}), and define (anomalous)
Taylor length scale associated with $B(\bfo,R_0)$ by
\begin{equation}\label{tau_def}
\tau_0=\left(\frac{R_{0}^{2}}{T}\frac{\bfe_0}{\varepsilon_0}\right)^{1/2}\;.
\end{equation}

\medskip

Henceforth, all the averages $\lgl\cdot\rgl_R$ are taken with
respect to $(K_1,K_2)$-covers at scale $R$.

\medskip

The following lemma will be a key technical ingredient in
establishing the energy cascade in the next section.

\begin{lem}\label{anom_flux_lem}
Let $\{B(\bfx_i,R)\}_{i=1,n}$ be a $(K_1,K_2)$-cover of
$B(\bfo,R_0)$ at scale $R$. Then
\begin{equation}\label{Phi_infty_est}
{K_{1}}\varepsilon_0\le\varepsilon_R\le K\varepsilon_0\;
\end{equation}
where $K>0$ is a constant depending only on $K_2$ and dimension of
the space (in $\mR3$ , one can choose $K=8^3K_2$).
\end{lem}

\begin{proof}
Let $\phi_{i}=\phi_{\bfx_{i},R}$ be a smooth cut-off function associated with
 $B(\bfx_{i},R)$
as described in (\ref{eta_def}).

\medskip

To prove the first inequality in (\ref{Phi_infty_est}), note that $\tilde{\phi}=
\sum_{i}\phi_{i}-\phi_{0}\ge 0$, and so the local energy inequality
(\ref{lee}) written for $\tilde{\phi}$,
\[
0\le \iint\frac{1}{2}\left[ (|\bfu|^2+2p)\bfu\cdot\nabla(\sum_{i}\phi_{i}-\phi_0)
+|\bfu|^2\partial_t(\sum_{i}\phi_{i}-\phi_0)\right]\,d\bfx\,dt\;,
\]
implies
\[\begin{aligned}
\varepsilon_0&=\frac{1}{T}\frac{1}{R_0^3}
\iint\frac{1}{2}\left[|\bfu|^2\partial_t\phi_{0} +(|\bfu|^2+2p)\bfu\cdot\nabla\phi_{0}\right]
\,d\bfx\,dt\\
&\le\frac{1}{T}\frac{1}{R_0^3}
\iint\frac{1}{2}\left[|\bfu|^2\partial_t\Bigl(\sum_{i}\phi_{i}\Bigr) +(|\bfu|^2+2p)\bfu\cdot\nabla(\sum_{i}
\phi_{i})\right]\,d\bfx\,dt\\
&\le \frac{1}{K_{1}}\frac{1}{T}\frac{1}{R^3}\frac{1}{n}\sum\limits_{i}\iint\frac{1}{2}
\left[|\bfu|^2\partial_t\phi_{i} +(|\bfu|^2+2p)\bfu\cdot\nabla\phi_{i}\right]\,d\bfx\,dt=
\frac{1}{K_{1}}\varepsilon_R\;,
\end{aligned}
\]
where we used the following consequence of the definition of a
$(K_1,K_2)$-cover,
\[\frac{1}{R_{0}^{3}}\le \frac{1}{K_{1}}\frac{1}{R^{3}}\frac{1}{n}\,.\]

\medskip

To prove the second inequality in (\ref{Phi_infty_est}),
 let $\{\bfx_{i_j}\}$ be a subset of
 $\{\bfx_i\}_{i=1,n}$ such that
interiors of the balls $B(\bfx_{i_j},2R)$ are pairwise disjoint.
Using (\ref{anomal-diss}), we obtain
\begin{equation}\label{lem_eq1}
\begin{aligned}
TR_0^3\,\varepsilon_0=&\iint\left[ (\frac{1}{2}|\bfu|^2+p)\bfu\cdot\nabla\phi_0
+\frac{1}{2}|\bfu|^2\partial_t\phi_0\right]\,d\bfx\,dt\;
\end{aligned}
\end{equation}
and
\begin{equation}\label{lem_eq2}
\begin{aligned}
\sum\limits_{j}\varepsilon_{\bfx_{i_j},R}=
&\iint\left[(\frac{1}{2}|\bfu|^2+p)\bfu\cdot\nabla(\sum\limits_{j}\phi_{i_j})
+ \frac{1}{2}|\bfu|^2\partial_t(\sum\limits_{j}\phi_{i_j})\right]\,d\bfx dt\,.
\end{aligned}
\end{equation}

In this scenario
\[
\tilde{\phi}=\phi_0-\sum\limits_{j}\phi_{i_j}\ge0\;;
\]
hence, by the local energy inequality (\ref{lee}),
\begin{equation}\label{lem_eq3}
\begin{aligned}
0\le&\iint(\frac{1}{2}|\bfu|^2+p)\bfu\cdot\nabla\tilde{\phi}\,d\bfx\,dt+
\frac{1}{2}\iint|\bfu|^2\partial_t\tilde{\phi}\,d\bfx\,dt\;.
\end{aligned}
\end{equation}
If we add relations (\ref{lem_eq2}) and (\ref{lem_eq3}) and then subtract
(\ref{lem_eq1}) we obtain
\begin{equation}\label{claim1}
\sum\limits_{j}\varepsilon_{\bfx_{i_j},R}\le TR_0^3\,\varepsilon_0\;.
\end{equation}

Let $\mathcal{L}$ be a cubic lattice inside $B(\bfo,R_0)$ with the
points situated at the vertices of cubes of side $R/2$. (Note that
this lattice can be chosen such that the number of points in it is
between $2^3(R_0/R)^3$ and $(4\pi/3)2^3(R_0/R)^3$.)

\medskip

Since the cover $\{B(\bfx_i,R)\}$ is a $(K_1,K_2)$-cover, each point
in $\mathcal{L}$ is contained in at most $K_2$ balls. Moreover, any
ball in the cover will contain at least one point from the lattice.

\medskip

If $\mathcal{L}'$ is sub-lattice of $\mathcal{L}$ with points at
vertices of cubes of side $4R$, then the interiors of balls of
radius $2R$ containing different points of $\mathcal{L}'$ are
pairwise disjoint, and thus if we denote by $B(\bfx_{i_p},R)$ a ball
from the cover $\{B(\bfx_i,R)\}$ containing the point
$p\in\mathcal{L}'$, by (\ref{lem_eq3}),
\[\sum\limits_{p\in\mathcal{L}'}\varepsilon_{\bfx_{i_p},R}\le TR_0^3\,\varepsilon_0\;.\]

Note that for each point $p\in\mathcal{L}'$ there are at most $K_2$ choices for
$B(\bfx_{i_p},R)$.
So
\[
\sum\limits_{i:B(\bfx_{i},R)\cap\mathcal{L}'
\not=\emptyset}\varepsilon_{\bfx_{i},R}\le K_2TR_0^3\,\varepsilon_0\;.
\]
Clearly $\mathcal{L}$  can be written as a union of $8^3=256$
sub-lattices $\mathcal{L}'_k$, $k=1,\dots,256$, each
$\mathcal{L}'_k$ having the same properties as $\mathcal{L}'$.
Thus,
\[
\sum\limits_{i=1}^{n}\varepsilon_{\bfx_{i},R}\le8^3K_2TR_0^3\,\varepsilon_0\;.
\]
Consequently,
\[\varepsilon_R=\frac{1}{T}\frac{1}{R^3}\frac{1}{n}\sum\limits_{i=1}^{n}\varepsilon_{\bfx_{i},R}\le
8^3K_2\left(\frac{R_0}{R}\right)^3\frac{1}{n}\,\varepsilon_0\le
8^3K_2\,\varepsilon_0\; .\]
\end{proof}

\begin{obs}
{\em Note that the \emph{defect in the local energy inequality} can
be interpreted either as the anomalous dissipation or as the
anomalous flux. The second interpretation was adopted -- in the
context of 3D viscous flows -- in \cite{DG10} which contains a proof
of the upper bound in the lemma interpreted as an upper bound on the
averaged anomalous fluxes.}
\end{obs}

\section{Energy Cascade}\label{balls}

Let $\{B(\bfx_i,R)\}_{i=1,n}$ be a $(K_1,K_2)$-cover of
$B(\bfo,R_0)$ at scale $R$. Definitions of $\varepsilon_{R}$ and
$\Phi_{R}$ -- (\ref{E_R_def}) and (\ref{Phi_R_def}) -- imply

\begin{equation}\label{ene-eq}
\Phi_{R}=\varepsilon_R-\frac{1}{n}\sum\limits_{i=1}^{n}\frac{1}{T}\frac{1}{R^3}\iint\frac{1}{2}
|\bfu|^2\partial_t\phi_i\,d\bfx\,dt\;
\end{equation}
where $\phi_i=\eta\psi_i$ and $\psi_i$ is the spatial cut-off on
$B(\bfx_i,2R)$ satisfying (\ref{eta_def}).

\begin{equation}\label{phi_t-bd}
|(\phi_i)_t|=|\eta_t\psi_i|\le C_0\frac{1}{T}\eta^{\delta}\psi_i\le
\frac{C_0}{T}\phi_i^{\delta}\,;
\end{equation}
hence,
\[\Phi_R\ge  \varepsilon_R - \frac{C_0}{T}\,\bfe_R.\]

The definition of $(K_1,K_2)$-covers at scale $R$ and the design of
$\psi_i$ paired with (\ref{e_R_def}) imply
\begin{equation}\label{e_R_e_ineq}
\bfe_R\le{K_2}\bfe_0\;,
\end{equation}
while Lemma \ref{anom_flux_lem} states
\begin{equation}\label{E_R_E_ineq}
\varepsilon_R\ge {K_1}\varepsilon_0.
\end{equation}

Consequently,
\begin{equation}\label{low_bd}
\Phi_R\ge  {K_1}\varepsilon_0 -\frac{C_0K_2}{T}\,\bfe_0 \ge
K_{1}\varepsilon_0\, \left(1-c_1\frac{\tau_0^2}{R_{0}^2}\right)
\end{equation}
with $c_1=C_0K_2/K_{1}$.

\medskip

Suppose that
\begin{equation}\label{scales_con}
\tau_0< \frac{\gamma}{c_1^{1/2}}R_0
\end{equation}
for some $0<\gamma<1$. Then, for any $R \le R_0$, (\ref{low_bd}) yields
\begin{equation}\label{lower_bd}
\Phi_R\ge{K_{1}}(1-\gamma^2)\varepsilon_0=c_{0,\gamma}\varepsilon_0\;
\end{equation}
where
\begin{equation}
c_{0,\gamma}={K_1}(1-\gamma^2)\;.
\end{equation}

\medskip

To obtain an upper bound on the averaged flux, we utilize the
estimates (\ref{phi_t-bd}) and (\ref{e_R_e_ineq}) in the identity
(\ref{ene-eq}) again, this time together with the upper bound in
Lemma \ref{anom_flux_lem} to obtain
\[\Phi_R\le  \varepsilon_R+\frac{C_0}{T_{0}}\bfe_R\le K\varepsilon_0+
C_0K_2\frac{1}{T}\,\bfe_0.\]
If the condition (\ref{scales_con}) holds for some $0<\gamma<1$,
then it follows that for any $R\le R_0$,
\begin{equation}
\Phi_R\le  K\varepsilon_0 +
\frac{C_0K_2\gamma^2}{c_1}\varepsilon_0\le
c_{1,\gamma}\varepsilon_0\;
\end{equation}
where
\begin{equation}
c_{1,\gamma}=K \left[1+\frac{C_0K_{2}\gamma^2}{c_1K}\right]=K
\left[1+\frac{K_1}{K}\gamma^2\right]\;.
\end{equation}

Thus we have proved the following.

\begin{thm}\label{balls_thm}
Assume that for some $0<\gamma<1$
\begin{equation}\label{scales_con_fin}
\tau_0< {\gamma}c\,R_0\;,
\end{equation}
where
\begin{equation}\label{c_con1}
c=\sqrt{\frac{K_{1}}{C_0K_2}}\;.
\end{equation}
Then, for all $R$,
\begin{equation}\label{inert_range}
0 < R\le R_0,
\end{equation}
the averaged energy flux $\Phi_R$ satisfies
\begin{equation}\label{ener_casc}
c_{0,\gamma}\varepsilon_0\le\Phi_R\le c_{1,\gamma} \varepsilon_0\;
\end{equation}
where
\begin{equation}\label{c_con2}
c_{0,\gamma}=K_{1}({1-\gamma^2})\,, \quad
c_{1,\gamma}=K \left[1+\frac{K_1}{K}\gamma^2\right]\;,
\end{equation}
and the average $\lgl\cdot\rgl_R$ is computed over a time interval
$(0,2T)$ and determined by a $(K_1,K_2)$-cover of $B(\bfo,R_0)$ at
scale $R$.
\end{thm}

As already noted in the previous section -- in the case the global
energy is non-increasing -- the positivity of the local flux
$\Phi_{\bfx_i, R}$ implies transfer of the kinetic energy around the
point $\bfx_i$ at scale $R$ from larger to smaller scales. Since we
assume no spatial homogeneity, the positivity of the flux expressed
in the theorem holds only in the averaged sense over the integral
domain. This may cause some uneasiness, as the transfer of the
kinetic energy around the point $\bfx_i$ at scale $R$ is, in
general, not necessarily dominantly local; this is due to the fact
that the pressure -- in terms of the velocity of the flow -- is
given by a non-local operator. In order to ensure existence of a
\emph{bona fide} (kinetic) energy cascade, we need \emph{locality}
in the sense of turbulence phenomenology, i.e., for the averaged
flux at the given scale to be well-correlated only with the averaged
fluxes at nearby scales, \emph{throughout the inertial range}. This
is in fact true, and is a simple consequence of the universality of
the averaged fluxes \emph{per unit mass} $\Phi_R$ displayed in
Theorem 3.1. More precisely, denoting by $\widetilde{\Phi}_R$ the
time-averaged ensemble average of the local fluxes $\Phi_{\bfx_i,
R}$,
\[
\widetilde{\Phi}_R=\frac{1}{T}\int \frac{1}{n}\sum\limits_{i=1}^{n}
\Phi_{\bfx_i, R}(t)\,dt = R^3 \Phi_R,
\]
the following holds.

\begin{thm}
Let $R$ and $r$ be two spatial scales within the inertial range
obtained in Theorem 3.1. Then,
\[
\frac{c_{0,\gamma}}{c_{1,\gamma}}\biggl(\frac{R}{r}\biggr)^3 \le
\frac{\widetilde{\Phi}_R}{\widetilde{\Phi}_r} \le
\frac{c_{1,\gamma}}{c_{0,\gamma}}\biggl(\frac{R}{r}\biggr)^3
\]
where the constants are the same as in Theorem \ref{balls_thm}.
\end{thm}

\begin{obs}
{\em In particular, in the dyadic case -- $r=2^kR$ -- both the
infrared and the ultraviolet locality propagate exponentially in the
dyadic parameter $k$. More precisely,
\[
\frac{c_{0,\gamma}}{c_{1,\gamma}} 2^{-3k} \le
\frac{\widetilde{\Phi}_R}{\widetilde{\Phi}_{2^kR}} \le
\frac{c_{1,\gamma}}{c_{0,\gamma}} 2^{-3k}.
\]
Note that the ultraviolet locality in the inviscid case is more
pronounced than in the viscous case as the cascade here continues
\emph{ad infinitum} ($r \to 0 \ \Leftrightarrow k \to -\infty$). }
\end{obs}

\begin{obs}\label{Rem_4.2}{\em
In the language of turbulence, the condition (\ref{scales_con_fin})
simply reads that the
(anomalous) Taylor \emph{micro-scale} computed over the
domain in view is smaller than the \emph{integral scale} (diameter
of the domain). }\end{obs}

\begin{obs}
{\em Unlike in the Navier-Stokes case (cf. \cite{DG10}), we do not
impose an explicit constraint on the length of the time interval
$T$. Implicitly, $T$ is a part of cascade condition
(\ref{scales_con_fin}) through the definition of Taylor scale
$\tau_{0}$ (\ref{tau_def}).
 }
\end{obs}

\section{Cascade near spatially isolated singularity}

Assuming certain \emph{geometric structure} of the singular set
leads to improved results on existence of the energy cascade in 3D
inviscid flows. In particular -- in the case of a \emph{spatially
isolated singularity} (the singular set being a curve
$\bfc=\bfc(t)$) -- it is enough to assume that the strict energy
inequality holds on some neighborhood of the singular curve. For
simplicity, we present the proof in the case the curve is a line
segment; the proof easily generalizes to the case of a smooth curve.

\medskip

The following lemma states that the anomalous dissipation is
constant on tubular neighborhoods of the singular
line segment.

\begin{lem}
Let $0 < R_1 < R_2, T > 0$, and let $\bfu$ be weak solution to the 3D Euler
equations on $B(\bfo, 2R_2) \times (0,2T)$ satisfying
the local energy inequality, smooth on
$\Bigl(B(\bfo, 2R_2) \setminus \{\bfo\}\Bigr) \times (0,2T)$.
Then
\[
 \varepsilon_{\bfo, R_1} = \varepsilon_{\bfo, R_2}.
\]
\end{lem}

\begin{proof}
Let $\phi_1$ and $\phi_2$ be test functions on $B(\bfo, 2R_1)$ and
$B(\bfo, 2R_2)$, equal to 1 on  $B(\bfo, R_1)$ and $B(\bfo, R_2)$,
respectively. Since $\bfu$ is smooth on the support of
$\phi_2-\phi_1$, integration by parts yields
\[
\iint\Bigl(\frac{1}{2}|\bfu|^2+p\Bigr)\bfu\cdot\nabla(\phi_2-\phi_1)\,d\bfx\,dt +
\frac{1}{2}\iint|\bfu|^2 \partial_t(\phi_2-\phi_1) \,d\bfx\,dt\  = 0;
\]
hence, $ \varepsilon_{\bfo, R_1} = \varepsilon_{\bfo, R_2}$.
\end{proof}

The general theorem on existence of 3D inviscid energy cascade --
Theorem 3.1. -- paired with the above lemma yields the following
result on existence of the cascade in the neighborhood of a singular
line.

\begin{thm}
Let $R, T > 0$, and let $\bfu$ be weak solution to the 3D Euler
equations on $B(\bfo, 2R) \times (0,2T)$ satisfying
the local energy inequality, smooth on
$\Bigl(B(\bfo, 2R) \setminus \{\bfo\}\Bigr) \times (0,2T)$. Assume
that the strict energy inequality holds on
$B(\bfo, 2R^*) \times (0,2T)$ for some $R^*$,
$0 < R^* \le R$, i.e.,
$\varepsilon_{\bfo,R^*} > 0$.
Then, there exists $R_0^*$,
$0 < R_0^* \le R$ such that the cascade condition
(\ref{scales_con_fin}), $\tau_0< {\gamma}c\,R_0$,
holds for any $R_0$, $0 < R_0 \le R_0^*$.
\end{thm}

\begin{proof}
Note that the cascade condition holds on $B(\bfo, 2R_0) \times (0,2T)$
for some $R_0$, $0 < R_0 \le R$ if and only if
\[
\iint \frac{1}{2}|\bfu|^2\phi_0 < \gamma^2c^2T \varepsilon_{\bfo, R_0}.
\]
However, by the above lemma,
\[
 \varepsilon_{\bfo, R_0} = \varepsilon_{\bfo, R^*} > 0.
\]
Since $\bfu$ is a weak solution with the locally finite energy,
the condition will hold for all sufficiently small $R_0$.
\end{proof}

As already noted in Introduction, a natural Onsager critical space
in this setting is $(L^3_t L^{4.5}_{\bfx})_{loc}$ (cf. \cite{Shvy09});
hence, a natural class of weak solutions to be considered here is
$(L^3_t L^\alpha_{\bfx})_{loc}, \, 3 \le \alpha < 4.5$.


\end{document}